\documentclass[reqno,11pt]{amsart}
\usepackage{amscd,amssymb, amsmath}
\usepackage{graphicx}
\usepackage[frenchb, russian, english]{babel}
\usepackage[utf8]{inputenc}

\newtheorem{theorem}{Theorem}[section]
\newtheorem{lemma}[theorem]{Lemma}

\theoremstyle{definition}
\newtheorem{definition}[theorem]{Definition}

\newtheorem{proposition}[theorem]{Claim}
\newtheorem{assumption}[theorem]{Assumption}

\newtheorem{example}[theorem]{Example}

\newtheorem{conjecture}[theorem]{Hypothesis}

\newtheorem{note}[theorem]{Note}

\DeclareRobustCommand{\divby}{%
  \mathrel{\text{ \vbox{\baselineskip.65ex\lineskiplimit0pt\hbox{.}\hbox{.}\hbox{.}} }}%
}

\begin{document}
\author{Lev Soukhanov \\ \tiny{National Research University Higher School of Economics}}
\title{$2$-Morse Theory And Algebra Of The Infrared}
\pagestyle{plain}
\maketitle
\begin{abstract}
We develop the formal analogue of the Morse theory for a pair of commuting gradient-like vector fields. The resulting algebraic formalism turns out to be very similar to the algebra of the infrared of Gaiotto, Moore and Witten (see [GMW], [KKS]): from a manifold $M$ with the pair of gradient-like commuting vector fields, subject to some general position conditions we construct an $L_{\infty}$-algebra and Maurer-Cartan element in it.

We also provide Morse-theoretic examples for the algebra of the infrared data.
\end{abstract}

\tableofcontents

\section{Introduction}

We describe a geometric setting of $2$-Morse theory, in which algebra of the infrared arises naturally. Some parts of the text depend heavily on [KKS], our main source of knowledge about algebra of the infrared theory.

The $2$-Morse theory is a theory of a pair of commuting vector fields $v$ and $w$ on a smooth compact manifold $M$. These vector fields should also be subject to various assumptions, which we will discuss in the next parts.

One could also consider a $k$-tuple of commuting vector fields, but we restrict ourselves to the case of a pair of vector fields for the sake of simplicity.

\begin{definition} \hfill \break
$0$-dimensional orbits of the pair of vector fields (i.e. vanishing points of both $v$ and $w$) are called \textbf{critical points}.
\\ $1$-dimensional orbits (i.e. trajectories of both $v$ and $w$) are called \textbf{critical edges}.
\\ $2$-dimensional orbits are called \textbf{leaves}.
\end{definition}

Now we briefly outline the results.

In the classic Morse theory, one defines the graded space $V$, generated by the critical points (with appropriate grading coming from index) and constructs from the geometry of trajectories an element $\phi \in Hom^1(V, V)$, called Morse differential, satisfying $\phi^2 = 0$. The cohomology of this differential coincide with cohomology of $M$.

The similar situation is present in the $2$-Morse setting: namely, one defines an $L_{\infty}$ algebra $F$ from the geometry of critical edges, and finds (from geometry of leaves) an element $\phi \in F$ satisfying Maurer-Cartan equation.

Moreover, $F$ is actually quasi-isomorphic to the Hochshchild cohomology of some category $C$, and deformation of this category by $\phi$ is, though under much stronger assumptions, seen to be equivalent to the dg-category of trajectories of the field $v$ - so, $2$-Morse theory is in a sense a description of the space of trajectories of the field $v$ in terms of Morse theory of a commuting vector field.

These constructions are in parallel (and the last one is motivated by) with exactly the same algebra appearing in the different geometric setting: $\zeta$-instanton equation in [GMW], [KKS]. \\

The paper is organised as follows: at first we recall the notions from the classic Morse theory and motivate $2$-Morse, then we set our geometric assumptions, then we prove the existence of $L_\infty$ algebra $F$ and Maurer-Cartan element $\phi$. After that, we compare the results with [GMW] and [KKS], and discuss possible interpretations of the universality lemma. In the last part, we provide two explicit examples, coming from $\mathbb{RP}^n$ and $\mathbb{CP}^n$.

Author would like to thank his scientific advisor, Andrei Losev, for introducing the problem and the main idea, would like to thank Mikhail Kapranov for discussions about Algebra of the Infrared and Morse theory, would like to thank Petr Pushkar, Dmitry Korb, Jim Stasheff, Misha Verbitsky for their interest at the early stages of this work.

The study has been funded by the Russian Academic Excellence Project '5-100'.

\section{Motivation}

Our main initial motivation for developing $2$-Morse theory was an attempt to understand the enumerative geometry of surfaces. It is well known that moduli space problem for the maps from (complex) surface $S$ to some algebraic variety $X$ in general does not enjoy perfect obstruction theory, and there is no known theory of virtual fundamental classes for these kind of problems. The idea was to somehow circumvent this issue by finding whether some kind of algebraic structure governs the combinatorics of these moduli spaces for the good cases (say, with $X = \mathbb{CP}^n$). The $2$-Morse theory is the toy case of ${{(\mathbb{C}^{\star})}^2}$-action on $X$, from which one could hope to use localization formula. We are yet unable to do it, though.

The newfound motivation is the fact that the combinatorial structure appearing in $2$-Morse is the same as algebra of the infrared of Gaiotto, Moore and Witten. Not only $2$-Morse provides simple examples of these kind of structures, this connection is quite curious also for the reason that moduli problem considered in [GMW] is actually $1$-\textit{dimensional} (governed by elliptic complex of length $2$).

\section{$1$-Morse-Bott theory}

The good point to introduce the notions from the next parts of the article is a $1$-dimensional case. To see some important principles which will hold also in $2$-dimensional case it is useful to consider Morse-Bott version of the theory. So, the setting is a manifold $M$, endowed with a vector field $v$ which is gradient-like (that is, $v$ is the gradient of some function with respect to some Riemannian metric) such that critical points form the collection of smooth submanifolds $M_i$, and the vector field is non-degenerate in the directions transverse to the $M_i$'s.

We also fix some ordering on $M_i$'s such that it agrees with the direction of downward flow (it is always possible to choose this ordering because of the gradient-like condition). We will assume that ordering by $i$'s is exactly this ordering.

We use the notation $C_{\bullet}(X)$ for the complex of polyhedral chains on the manifold $X$. In what follows, we will actively use intersection of chains, so the operations on the complexes below won't be defined everywhere. Additional care is needed if one wants to explicitly use these algebras on dg-level - this is explicitly done, for instance, in the paper [BH].

Let us denote by $V = \bigoplus V_i = \bigoplus C_{\bullet} (M_i)$ the complex of chains of the collection $M_{i}$. It is endowed with the natural filtration by the order of $i$. We denote as $B = End(V) = \underset{i<j}\bigoplus Hom(V_i, V_j)$ the dg-algebra of oriented endomorphisms of the complex $V$. We call it $B$ because of the ''border'' mnemonic: the elements of complex (naturally quasiisomorphic to) $B$ are chains in the space of pairs of critical points, which can be thought as the ''borders'' of Morse trajectories. We also shift $V_i$'s grading by an index of the corresponding critical submanifold.

Consider also the space of trajectories $M_{ij}, i<j$. Under standard assumptions of generality it is a manifold with angles, admitting the following description of the boundary: \[\partial M_{ik} = \bigcup_{i<j<k} M_{ij} \underset{M_j}{\times} M_{jk}\]

We denote $T = \underset{i<j}{\bigoplus} T_{ij} = \underset{i<j}\bigoplus C_{\bullet}(M_{ij})$. The complex $T$ has a structure of (partially defined) dg-algebra with respect to the following composition: \[t_{\star}s^{\star}:T_{ij} \otimes T_{jk} \rightarrow T_{ik}\] is defined with the help of the following diagram: \[M_{ij} \times M_{jk} \underset{s}{\hookleftarrow} M_{ij} \times_{M_j} M_{jk} \underset{t}{\hookrightarrow} M_{ik}\]

where $s$ denotes natural inclusion of fibered product into product and $t$ denotes composition map of trajectories from the description of the boundary.

There is also a natural action of $T$ on $V$: \[T_{ij} \otimes V_i \rightarrow V_j\] The map $T_{ij} \rightarrow B_{ij}$ is a push-forward along the boundary map $M_{ij} \rightarrow M_i \times M_j$.

Denote as $\phi_{ij} \in T_{ij}$ the collection of fundamental chains of $M_{ij}$ (these are manifolds with angles, so there is a fundamental chain),

\[\phi = \bigoplus_{i<j} \phi_{ij}\]

Then, the description of boundary of the space of trajectories yields the following Maurer-Cartan equation:

\[\partial \phi_{ik} = \sum_{i<j<k} \phi_{ij} \circ \phi_{jk} \]
\[\partial \phi = \phi \circ \phi \]

One can also consider the projection of this element in $B$, which will also denote as $\phi$. The deformed differential $\partial - \phi$ on $V$ calculates the homology of $M$.

In the case of regular Morse theory the only component of $\phi$ which does not vanish under projection to $B$ component of (topological) degree zero - because $B$ is interpreted as the discrete space of pairs of critical points it has no polyhedral chains of higher degree. The differential on $V$ also vanishes, so the Maurer-Cartan equation becomes $\phi^2 = 0$ and Morse differential is recovered.

\section{$2$-Morse theory and geometric assumptions}

In this part we describe the analogous picture of $2$-Morse theory and various geometric assumptions. We want to point out that while most of these assumptions are analogous to some kind of ''generality'' conditions in $1$-Morse theory, we do not know any good way of imposing these conditions by the small perturbation. We have also some kind of no-go example for one of these conditions, which will be presented later.

\begin{assumption}[Gradient condition] $v = grad(f_v)$, $w = grad(f_w)$ for some functions $f_v$, $f_w$ w.r.t. to some Riemannian metric.
\end{assumption}

\begin{assumption}[Morse critical points condition] Critical points are isolated. Each one admits linearization in a local chart, at which both fields are linear diagonal vector fields.
\end{assumption}

\begin{assumption}[Critical edges condition] Critical edges are isolated.
\end{assumption}

Two previous conditions can be relaxed. Relaxing 4.3 to the Morse-Bott-like condition is mostly harmless and is also needed if one considers actions of $(\mathbb{C}^{\star})^2$, relaxing 4.2 will most likely lead to the notion of extended system of coefficients from [KKS].

\begin{definition} ($2$-dimensional) \textbf{toric fan} is a subdivision of a plane into a finite collection of angles (each less than $\pi$).
\end{definition}

\begin{definition}\textbf{Toric domain} is a compactification of $\mathbb{R}^2$ governed by a $2$-dimensional toric fan, which admits $\mathbb{R}^2$-action. This is a standard construction from toric geometry, we recall now how it works.

It is an $n$-gon topologically (where $n$ is a number of edges of the fan): interior of each angle corresponds to one vertex (and all $1$-parametric orbits going in the directions of this angle end in this vertex). Edges of the fan correspond to the $1$-dimensional orbits of $\mathbb{R}^2$, connecting the vertices.

Real part of $2$-dimensional toric variety corresponding to this fan is a union of 4 $\mathbb{R}_{>0}^2$-orbits, each one being a toric domain (under exponential identification of $\mathbb{R}$ and $\mathbb{R}_{>0}$).
\end{definition}

\begin{note} Closure of any leaf is a toric domain, with boundary on the union of critical edges.
\end{note}

\begin{definition} \textbf{Direction} of the oriented edge is the vector $(a, b) \in \mathbb{R}^2$ such that $bv - aw$ vanishes on the edge and $av + bw$ goes along the orientation of the edge.
\end{definition}

\begin{definition} \textbf{Frame} is a cyclic sequence of critical points connected by critical edges such that their directions are cyclically ordered. Frames are possible boundaries for the leaves. The vector space generated by frames (or chain complex on the space of frames in the case of Morse-Bott version) is denoted $F$ and is an analogue of $B$ from the previous part.
\end{definition}

For a frame $f$, we denote as $M_f$ the variety of leaves with the boundary $f$.

\begin{definition} We define \textbf{expected dimension} of $M_f$ as follows: consider the (smooth) normal bundle of an individual leaf of $M_f$. It is well-defined on a boundary of a leaf. This normal bundle also admits equivariant $\mathbb{R}^2$-action. Invariant Euler characteristic of this bundle w.r.t. to this action is called an expected dimension.
\end{definition}

\begin{assumption} $M_f$ has expected dimension.
\end{assumption}

\begin{note} This condition in regular Morse theory is guaranteed by the generality of the vector field. In other parts of enumerative geometry, virtual fundamental classes are used instead. We have no idea if it is possible to provide some kind of general position argument for this to hold and it is also unclear how to construct virtual fundamental chains because of the presense of second cohomology group of the normal sheaf.
\end{note}

The $k$-dimensional equivariant bundle over the toric domain admits the following construction: for every vertex $v_i$ define the set with multiplicities on the (dual) plane $S_i = \{p_{i_1}, ..., p_{i_k}\}$, corresponding to the characters of an action of $\mathbb{R}^2$ on a fiber over vertex $v_i$. For two neighboring vertices $v_i$ and $v_{i+1}$ consider the projection of the corresponding sets on the edge of the toric fan. They coincide (it is clear from analysis of the restriction of the bundle on the corresponding $1$-dimensional orbit).

If one connects these sets $S_i$ by the collections of lines, orthogonal to the edges, one gets a piecewise-linear curve (actually, homological $1$-chain) corresponding to the bundle. We call it \textbf{characteristic curve}. If characteristic curve goes through zero, we deform the corresponding segment slightly to go around zero counterclockwise.
\begin{center}
\includegraphics[trim={1cm 0 1cm 0}, clip, width=7cm]{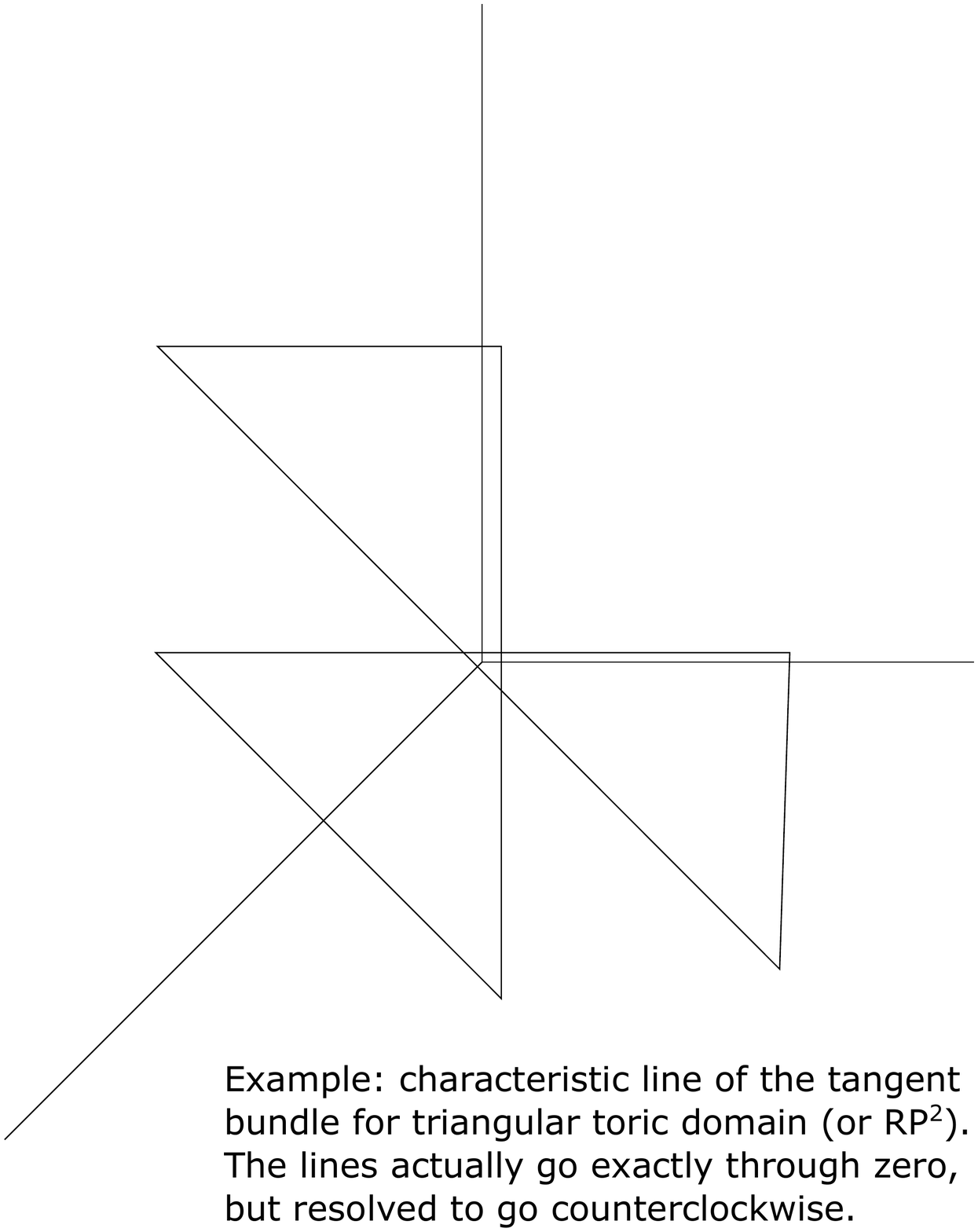}
\end{center}

\newpage

\begin{lemma} For a piecewise linear curve $\gamma$, comprised of the cyclically ordered collection of segments $\gamma_i$ on a plane the index (winding number) $I(\gamma)$ around zero satisfies

\[I(\gamma) = k - \sum_i a_i + \sum b_i\]

where $k$ is the turning number of $\gamma$

$a_i = 1$ if $\gamma_i$, prolonged to an oriented straight line, has zero in the left half-plane w.r.t. to the orientation of the line, and $0$ otherwise

$b_i = 1$ if a pair of segments $\gamma_i$, $\gamma_{i+1}$, prolonged to an infinite oriented angle (we prolong free ends of the segments to the infinite rays) has zero in the left part of the plane w.r.t. to the orientation of the angle, and $0$ otherwise
\end{lemma}

\begin{center}
\includegraphics[trim={0 11cm 0 2cm}, clip, width=9cm]{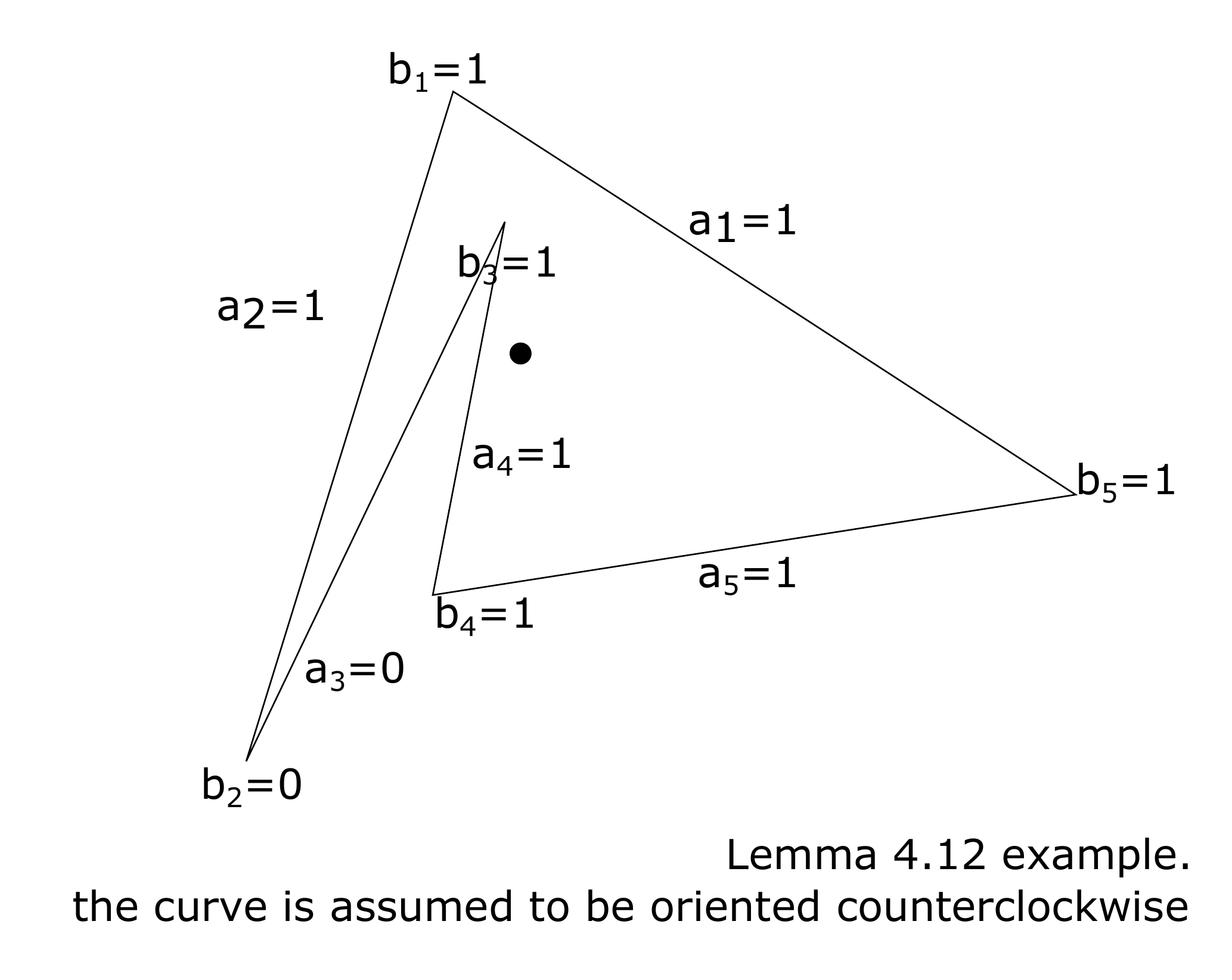}
\end{center}
\newpage

\begin{proof} Clearly, the sum does not change if one breaks up any segment into two parts. It is also well known that any two curves of a given turning number are connected by piecewise-linear homotopy without turning points (points where segment $\gamma_{i+1}$ goes backwards along the segment $\gamma_i$). If homotopy chosen is general enough, the prolongations of $\gamma_i$'s will go through zero one by one (these are codimension $1$ walls in the space of curves).

The formula is clearly true for curves that always go around zero counterclockwise ($a_i = b_i = 1$). It is enough to check how the RHS and LHS changes when the prolongation of some segment $\gamma_i$ goes through zero. There are two cases.

\textbf{Case 1.} Suppose $\gamma$ moves through the situation where segment $\gamma_i$ contains zero. Then, $a_i$ changes by $1$, and two neighboring terms, $b_i$ and $b_{i+1}$ also change by 1. The total sum changes by one (exactly as $I(\gamma)$).

\textbf{Case 2.} Suppose $\gamma$ moves through the situation where prolongation of a segment $\gamma_i$ contains zero, but not $\gamma_i$. Then, $a_i$ changes by $1$ and \textit{one} of the neighboring terms also. The total sum does not change.
\end{proof}

\begin{lemma} Invariant Euler characteristic of a bundle over a toric domain is an index of it's characteristic curve around $0$.
\end{lemma}

\begin{proof} Toric domain admits the acyclic covering using the standard coordinate charts $U_i$ corresponding to the vertices. We denote by $U$ the open part of the toric domain. Then, any invariant sheaf $\mathfrak{F}$ admits the following (Cech-like) resolution: \[\mathfrak{F} \rightarrow \bigoplus_i \mathfrak{F}|_{U_i} \rightarrow \bigoplus_i \mathfrak{F}|_{U_i \cap U_{i+1}} \rightarrow \mathfrak{F}|_U\] Now, one should use the invariant sections functor and calculate an Euler characteristic, which amounts exactly to the previous combinatorial lemma.
\end{proof}

\begin{note}It also should be noted that the characteristic curve is additive over short exact sequences of equivariant bundles: the characteristic curve of the middle term of an exact sequence is just $1$-chain sum of the characteristic curves of the left and right terms.
\end{note}

\section{Description of the structure}

\begin{assumption} Consider the set $A$ on a plane: \[A = \{(f_v(p_1), f_w (p_1)), ..., (f_v(p_n), f_w (p_n))\}\] $p_i$'s being the critical points. This set $A$ is assumed to be in general position.
\end{assumption}

It is not clear, whether it is always possible to deform $A$ to be in general position. While we do not know counterexamples, it is currently not clear for us.

In order to describe the boundary of the space $M_f$ in the general position, one needs to find possible degenerations occuring in the codimension $1$.  The degeneration of the leaf is the union of leaves, which is still a disk topologically. The corresponding net of frames will be called \textbf{subdivision}. This net projects onto $A$ as the subdivision of a convex polygon with vertices in $A$ into convex parts.

\begin{definition} \textbf{Graph with directions} is the following abstract data: graph with the special vertex ''infinity'', with the direction fixed for every oriented edge. These directions are subject to the following properties:

1) The change of orientation of the edge changes the direction to the opposite

2) Directions of outgoing edges of every vertex (including infinity) form a toric fan
\end{definition}

For every frame subdivision there is a dual graph with directions: with finite vertices corresponding to the frames of the subdivision, edges corresponding to the edges of between frames, outgoing edges corresponding to the outer edges of the subdivision, directions orthogonal to the directions of the edges.

We also recall the notion of the web, occuring in the works [GMW], [KKS].

\begin{definition} \textbf{Web} is a graph drawn on the euclidean plane with linear edges (and also some edges outgoing to infinity), each of its vertices corresponding to a convex fan, and, also, the set of outgoing edges also corresponds to a convex fan. Two webs are called equivalent if they realize the same graph with directions, that is: there is a bijection between vertices which preserves the directions of edges.
\end{definition}

\begin{definition} The graph with directions is called \textbf{realizable} if there is a web corresponding to it (in an obvious way). The dimension of the space of representations is an important invariant of the web - for instance, web is called \textbf{rigid} if there exists only one representation up to translations and dilation.
\end{definition}

The \textbf{space of realizations} of a given graph with directions forms a polyhedral cone in the configuration space of vertices: each edge defines a linear equation (angle of the edge) and a linear inequality (orientation of the edge).

\begin{definition} The \textbf{expected dimension} of the space of realizations equals $2v - e - 2$ (we quotient out parallel translations but not dilations. $v$ is the number of vertices, $e$ is the number of internal edges).
\end{definition}

\begin{definition}We will call the graph \textbf{semi-realizable}, if it is either realizable, or admits the following inductive description: a web with maybe some subgraphs shrinked into the points (we will call these almost-realizations), where shrinked subgraphs are also semi-realizable.
\end{definition}

\begin{lemma} \textbf{Semi-realizability} is a closed property over the space of correct fans of vertices.
\end{lemma}

\begin{proof} The existence of an almost-realization is a closed property, because it amounts to the existence of a non-zero vector in the closed polyhedral cone. Semi-realizability is almost-realizability of any subgraph.
\end{proof}

\begin{lemma} Assumption 5.1 implies that the dual graphs of the subdivisions of convex polygons with vertices in $A$ which are semi-realizable are actually realizable.
\end{lemma}

This seems to be a well known fact in the dual language of admissible subdivisions. We postpone the proof until the next section.

\begin{lemma} Suppose the leaf $L$ is degenerated to the union of leaves $L_1, ..., L_k$. Then, its dual graph is realizable.
\end{lemma}

\begin{proof} One constructs the web as the following limit. Consider the family of the leaves $L$, degenerating, and pick a family of points $p_1, ..., p_k$ in such a way that in the limit they land on the leaves $L_1, ..., L_k$. This collection of points are to be pulled back $\mathbb{R}^2$, and gives one the vertices of the web, and it is clear that the directions of the edges tend to the desired ones. By the previous lemmas, it is semi-realizable and, hence, realizable.

This is also quite parallel to the description of physical argument we learned from [GMW].
\end{proof}

\begin{center}
\includegraphics[trim={0 2cm 0 0}, clip, width=7cm]{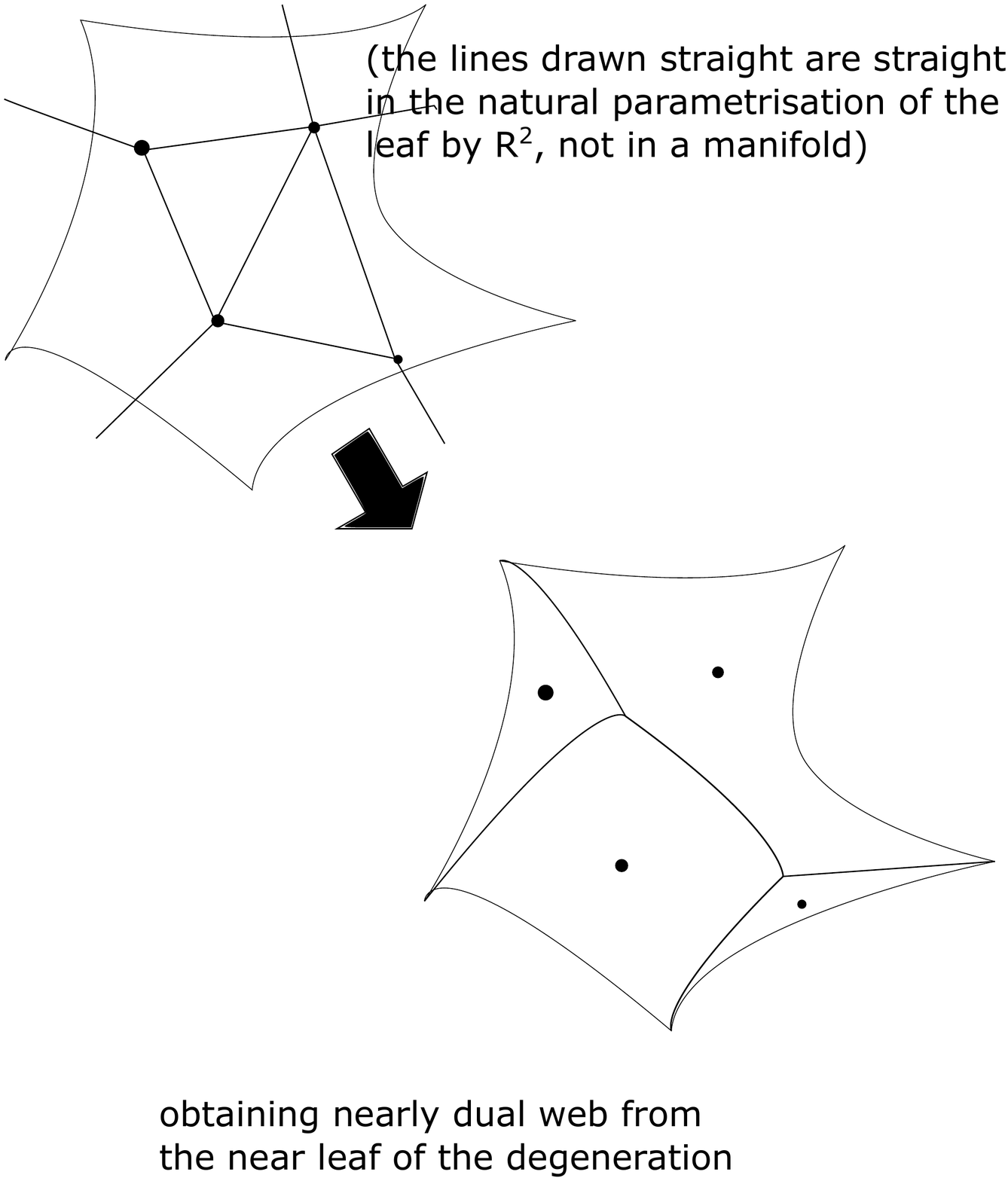}
\end{center}

Now, in order to proceed with the description of the boundary of the space $M_f$, one needs to find degenerations of codimension $1$: \[\partial M_{f} = \bigcup M_{f_1} \times ... M_{f_k} \] over all subdivisions of the frame $f$ into the frames $f_1, ..., f_k$, occuring in the general degeneration of leaf in codimension $1$.

\begin{lemma} Denote by $s$ the number of internal edges in the subdivision. The \textbf{expected codimension formula} states that the expected codimension $C$ of this subdivision coincides with the expected dimension of realizations of the dual graph, namely, it equals \[C = 2k-s-2\]
\end{lemma}

\begin{proof} One needs to calculate the difference between invariant Euler characteristic of the normal sheaf of $l$ and invariant Euler characteristic of the union of $l_1, ..., l_k$. It is clear that invariant Euler characteristic of the restriction of $TM$ is additive in that regard: the characteristic curve of $TM|_l$ equals the sum of characteristic curves of $TM|_{l_i}$'s as the sum of $1$-chains. Hence, one needs to calculate the jump of the invariant Euler characteristics of tangent sheaves, which is interpreted as the expected dimension of the automorphism group.
\end{proof}

\begin{note} One can notice that this formula gives $0$ or even negative values for some degenerations. At the initial stages of this work we thought this must be attributed to the existence of the second cohomology of the normal sheaf, but now we know the examples where it is not related. The corresponding webs never occur under assumption 5.1. The general case is interesting, its geometry in the complex case is discussed in [A].
\end{note}

\begin{lemma} Consider the set of operations $l_k: F \otimes \underset{\text{k times}}{ ... }\otimes F \rightarrow F$ defined as a sum over all decomposition of a frame $F$ into $k$ pieces in such a way that the corresponding web exists and expected codimension is $1$. These operations satisfy $L_\infty$ properties, and so $F$ is an $L_\infty$-algebra. \end{lemma}

\begin{proof} The proof is contained in [GMW] (web differential square is zero) and recast on the dual language in [KKS].
\end{proof}

\begin{lemma} Consider an element $\phi \in F$ defined by analogy to the $1$-dimensional case - counting the number of leaves passing through the frame if it is discrete. It is a Maurer-Cartan element.
\end{lemma}

\begin{proof}Analogous to the $1$-dimensional case, this fact follows from the description of codimension $1$ boundary.\end{proof}

\section{Relation to secondary polytopes / Comparison with [KKS]}

\begin{definition} The subdivision of a polygon is called admissible if there is a piecewise-linear concave function defining it.
\end{definition}

\begin{lemma} Admissibility is equivalent to the existence of the normal web.
\end{lemma}

\begin{proof} For the piecewise-linear concave function on the polygon $P$ with vertices in $A$, linear on the subpolygons $P_1$, ..., $P_k$, let us choose vertices of the web in the dual space to be the corresponding linear functions $f_i$'s on these polygons. Directions of the edges are guaranteed by the vanishing of $f_i - f_j$ on the edge between $P_i$ and $P_j$, positivity of the edges follows from concavity of the function.

In the opposite direction, web gives the linear parts of $f_i$'s, and one only needs to add additive constants in such a way to glue them into one piecewise-concave function. This is always possible, because the obstruction lies in a group $H^1(P, \mathbb{R}) = 0$.
 \end{proof}

This also identifies the Lemma 5.8 as the fact from [KKS] that in general position of the set $A$ admissible subdivisions have factorization property.

\begin{proposition} Under the previous conditions the space of critical edges gives rise to the factorization system on the secondary polytope of $A$ (as in [KKS]). The fundamental cycle of the space of leaves gives the Maurer-Cartan element in the corresponding $L_\infty$ algebra $F$.
\end{proposition}

While we described the structure analogous to the $1$-dimensional Morse theory, it is important to note that $F$ is the analogue of $B = End(V)$, but analogue of $V$ itself is yet to be spotted. From the paper [KKS] we learned \textbf{universality lemma} stating that $F$ can be interpreted as the Hochschild cohomology of the algebra of concave edge-paths, which will be denoted as $F_\curvearrowright$ in the further sections, so it is natural to consider the deformation of $F_\curvearrowright$ by the Maurer-Cartan element. So, $F_\curvearrowright$ is our candidate for the analogue of $V$, and its deformation is our candidate for the analogue of the Morse complex. We discuss possible geometric interpretation of these objects in the next part.

\section{Possible interpretation of universality lemma}

This part is an informal discussion on the meaning of the algebra of concave paths $F_\curvearrowright$ and its deformation.

\begin{definition} We denote by $F_\curvearrowright$ the algebra of edge-paths which are concave on the polarization plane. The multiplication is composition if the resulting path is concave and zero otherwise. It also to be endowed with some natural grading which we will not discuss explicitly.
\end{definition}

\begin{lemma}[Follows from universality lemma] The oriented bar- \linebreak construction of $F_\curvearrowright$ is an algebra of convex paths, defined in the same fashion. We denote it as $Bar(F_\curvearrowright)$
\end{lemma}

\begin{proof}
Consider the oriented bar-construction of $F_\curvearrowright$. It is an algebra of edge paths, going from left to right on the polarization plane and subdivided into concave pieces by the $\otimes$ symbol located in some internal vertices. The differential is a sum over removals of one $\otimes$ symbol. It can easily be seen that the subcomplex corresponding to the single edge-path is acyclic unless the pass is convex and $\otimes$ symbols are arranged on its every vertex.
\end{proof}

Now assume that the horizontal vector field $w$ is general in the sense of Morse theory (\textit{this assumption is almost never satisfied and this is the reason why we consider this more as a plausible explanation than a theorem}). Then one might try to construct the finite-dimensional model of the trajectories dg-algebra $B$ using the Morse theory of the vector field $w$ on the space of trajectories. The set of critical points of $w$ on the space of trajectories is easily identified with the set of edge paths going from the left to the right. Now one needs to remember that the space of trajectories is not a manifold but a manifold with angles, so the critical points on the boundary will be endowed with the complex which models the $1$-point compactification of the downward flow gradient variety. Inspection shows that the downward flow gradient variety of the edge-path going from the left to the right is of the form $\mathbb{R}^k \times \mathbb{R}_{>0}^{s_1} \times ... \times \mathbb{R}_{>0}^{s_n}$, where $s_1, ..., s_n$ are the lengths of concave parts of the path minus one. So, the bar construction complex is actually a model of the critical points on the space of trajectories, and its deformation should correspond to the Morse complex model of the space of trajectories.

\begin{conjecture}
$F_{\curvearrowright}$ deformed by Maurer-Cartan element is a model of dg-algebra $B$ of trajectories of the vector field $v$.
\end{conjecture}

\section{Questions and examples}

There are multiple issues with this model, the first and foremost being the lack of general position argument. We know no way of setting two given commuting vector fields in the position general enough for these constructions to hold. While for some assumptions it seems possible to achieve this, it is definitely not possible for the last argument because of the following no-go example:

\begin{example}
Consider the real toric variety with normal toric fan with $5$ edges. It is topologically the connected sum of $3$ copies of $\mathbb{RP}^2$. Any $2$-dimensional orbit is a $5$-gon. Linear combination of the fields $av + bw$ on a $5$-gon will have $1$ source point, $1$ sink point and $3$ saddle points. Hence, any such linear combination will have a gradient trajectory between two saddle points, which means they are not in general position in a sense of Morse theory.
\end{example}

Many assumptions of the model might be relaxed. For example, it should be interesting to consider algebraic varieties with $(\mathbb{C}^*)^2$-action, for which the space of critical edges will not be discrete anymore. It will be the kind of $2$-Morse-Bott theory.

We conclude with two examples, defined for any set $A$. For an $n+1$-point set $A$ one can define a diagonal action of $\mathbb{R}^2$ on an $n+1$-dimensional vector space: action on the $i$'th coordinate is $s_i \rightarrow e^{x_i t+y_i w} s_i$. This allows to construct an action on $\mathbb{RP}^n$ and $\mathbb{CP}^n$, second is slightly out of scope of the previous exposition, but the same algebraic structure is present.

\begin{example}
For $\mathbb{RP}^n$, there are $2$ critical trajectories for any pair of critical points $p_i$ and $p_j$ (corresponding to the two ways of connecting points along $\mathbb{RP}^1$). We denote them as $c_{ij}^{0}$ and $c_{ij}^{1}$, the first one is the one allowing the preimage with non-negative coordinates.

Then, the Maurer-Cartan element is the following: for any triangle without internal points $(i, j, k)$ it is
\[\Phi_{ijk} = \sum_{x+y+z \divby 2} c_{ij}^x c_{jk}^y c_{ki}^z\]
and vanishes in all other cases.
\end{example}

\begin{example}
For $\mathbb{CP}^n$, there are $S^1$-family of critical trajectories for any pair of critical points $p_i$ and $p_j$ (corresponding to the ways of connecting $0$ and $\infty$ along $\mathbb{CP}^1$). Thus, the homology space of this family is $1+1$-dimensional, we still denote the basis as $c_{ij}^{0}$ and $c_{ij}^{1}$. It is also endowed with (odd) Poincare pairing.

Then, the Maurer-Cartan element is the following: for any triangle without internal points $(i, j, k)$ it is
\[\Phi_{ijk} = \sum_{x+y+z = 2} c_{ij}^x c_{jk}^y c_{ki}^z\]
and vanishes in all other cases.
\end{example}

Both of these examples are fairly simple. They seem to also satisfy categorical wall-crossing formula from [GMW]. We currently aim to construct more elaborate examples from the enumerative geometry of projective spaces.

\end{document}